\documentclass[12pt,reqno]{amsart}
\usepackage{color}
\usepackage{tikz}
\usepackage{url}
\usepackage{float}

\newtheorem{Theorem}{Theorem}[section]
\newtheorem{Lemma}[Theorem]{Lemma}
\newtheorem{Corollary}[Theorem]{Corollary}
\newtheorem{Proposition}[Theorem]{Proposition}

\newtheorem{Example}[Theorem]{Example}

\newtheorem{Definition}[Theorem]{Definition}
\newtheorem{Conjecture}[Theorem]{Conjecture}

\def\qed{\ifhmode\textqed\fi
	\ifmmode\ifinner\hfill\quad\qedsymbol\else\dispqed\fi\fi}
\def\textqed{\unskip\nobreak\penalty50
	\hskip2em\hbox{}\nobreak\hfill\qedsymbol
	\parfillskip=0pt \finalhyphendemerits=0}
\def\dispqed{\rlap{\qquad\qedsymbol}}

\def\p{\mathfrak{p}}

\def\m{\mathfrak{m}}

\def\v{\textup{v}}
\def\ZZ{\mathbb{Z}}
\def\soc{\textup{soc}}

\def\HS{\textup{HS}}
\def\set{\textup{set}}

\def\depth{\textup{depth}}
\def\reg{\textup{reg}}
\def\lcm{\textup{lcm}}
\def\supp{\textup{supp}}

\def\Tor{\textup{Tor}}
\def\Ass{\textup{Ass}}

\usepackage{lipsum}

\begin{document}
	
	\title{Polymatroidal ideals\\ and their asymptotic syzygies}	
	\author{Antonino Ficarra, Dancheng Lu}
	
	\address{Antonino Ficarra, BCAM -- Basque Center for Applied Mathematics, Mazarredo 14, 48009 Bilbao, Basque Country -- Spain, Ikerbasque, Basque Foundation for Science, Plaza Euskadi 5, 48009 Bilbao, Basque Country -- Spain}
	\email{aficarra@bcamath.org,\,\,\,\,\,\,\,\,\,\,\,\,\,antficarra@unime.it}
	
	\address{Dancheng Lu, School  of Mathematical Sciences, Soochow University, 215006 Suzhou, P.R.China}
	\email{ludancheng@suda.edu.cn}
	
	\thanks{
	}
	
	
	
	\begin{abstract}
		Let $I$ be a polymatroidal ideal. In this paper, we study the asymptotic behavior of the homological shift ideals of powers of polymatroidal ideals. We prove that the first homological shift algebra $\text{HS}_1(\mathcal{R}(I))$ of $I$ is generated in degree one as a module over the Rees algebra $\mathcal{R}(I)$ of $I$. We conjecture that the $i$th homological shift algebra $\text{HS}_i(\mathcal{R}(I))$ of $I$ is generated in degrees $\le i$, and we confirm it in many significant cases. We show that $I$ has the $1$st homological strong persistence property, and we conjecture that the sequence $\{\text{Ass}\,\text{HS}_i(I^k)\}_{k>0}$ of associated primes of $\text{HS}_i(I^k)$ becomes an increasing chain for $k\ge i$. This conjecture is established when $i=1$ and for many families of polymatroidal ideals. Finally, we explore componentwise polymatroidal ideals, and we prove that $\text{HS}_1(I)$ is again componentwise polymatroidal, if $I$ is componentwise polymatroidal.
	\end{abstract}
	
	\maketitle
	\section{Introduction}\label{sec1}
	
	Let $S=K[x_1,\dots,x_n]$ be the standard graded polynomial ring over a field $K$, and let $I\subset S$ be a monomial ideal. In \cite{FQ,FQ1}, the first author and Qureshi introduced a new aspect to the theory of syzygies of powers of $I$. Recall that the $i$th \textit{homological shift ideals} \cite{HMRZa} of $I$ is defined as
	$$
	\HS_i(I)\ =\ ({\bf x^a}:\ \Tor_i^S(K,I)_{\bf a}\ne0),
	$$
	where ${\bf x^a}=\prod_{i=1}^{n}x_i^{a_i}$ if ${\bf a}=(a_1,\dots,a_n)\in\ZZ_{\ge0}^n$. Note that $\HS_0(I)=I$. Homological shift ideals encode critical multigraded data arising from the syzygies of $I$, see \cite{Bay019, Bay2023, BJT019, CF1, CF2, F2, FPack1, F-SCDP, FH2023, FQ, FQ1, HMRZa, HMRZb, LW, RS, TBR24}. A major motivation to study homological shift ideals comes from the Bandari-Bayati-Herzog conjecture \cite{Bay019, HMRZa}.
	
	It was conjectured in 2012 by Bandari, Bayati and Herzog that the homological shift ideals preserve the polymatroidal property. Polymatroidal ideals, which are the algebraic part of discrete polymatroids \cite[Chapter 12]{HHBook}, constitute one of the most distinguished classes of monomial ideals. Recently, Cid-Ruiz, Matherne and Shapiro settled the Bandari-Bayati-Herzog conjecture \cite[Theorem A(iii)]{CMS}.
	
	\begin{Theorem}\label{Thm:CMS}
		Let $I\subset S$ be a polymatroidal ideal. Then $\HS_i(I)$ is polymatroidal.
	\end{Theorem}
	
	Inspired by the Bandari-Bayati-Herzog conjecture, in \cite{FQ} the first author and Qureshi studied the problem of determining when the following $K$-algebra, which is called the $i$th \textit{homological shift algebra} of $I$,
	$$
	\HS_i(\mathcal{R}(I))\ =\ \bigoplus_{k\ge1}\HS_i(I^k),
	$$
	is a finitely generated module over the Rees algebra $\mathcal{R}(I)=\bigoplus_{k\ge0}I^k$ of $I$. While this is not the case in general \cite[Example 1.1]{FQ}, it holds true when $I$ has linear powers \cite[Theorem 1.4]{FQ} and it holds for the first homological shift algebra of any edge ideal \cite[Corollary 2.4]{FQ1}. When $\HS_i(\mathcal{R}(I))$ is a $\mathcal{R}(I)$-module, many invariants behave asymptotically well \cite[Theorem 2.1]{FQ}. For instance, $\Ass\,\HS_i(I^k)$ and $\depth\,S/\HS_i(I^k)$ became constant for large $k$, and $k\mapsto\reg\,\HS_i(I^k)$, $k\mapsto\v(\HS_i(I^k))$ became linear functions for large $k$. See Section \ref{sec2} for more details.
	
	Polymatroidal ideals have linear powers, and so their homological shift algebras are finitely generated modules over the Rees algebra. In this paper, employing the fundamental Theorem \ref{Thm:CMS}, we deeply study these algebras and the asymptotic behavior of the invariants of $\HS_i(I^k)$ with $I\subset S$ a polymatroidal ideal.\vspace*{0.1cm}
	
	The paper is organized as follows. Sections \ref{sec2} and \ref{sec3} summarize basic material on homological shift algebras and polymatroidal ideals. Let $I\subset S$ be a polymatroidal ideal. In Section \ref{sec4} we prove that $\HS_1(\mathcal{R}(I))$ is generated in degree one as a $\mathcal{R}(I)$-module (Theorem \ref{Thm:HS1-fg}). Based on this result, and several experiments, we expect that $\HS_i(\mathcal{R}(I))$ is generated in degrees $\le i$ (Conjecture \ref{Conj:GenDeg}). This conjecture holds for $i=1$. At the moment we do not have a general strategy to tackle this conjecture. In Section \ref{sec5}, we establish it for: (a) principal Borel ideals, (b) ideals satisfying the strong exchange property, and (c) matroidal edge ideals. To prove the case (b), we use the concept of M\"obius function of a polymatroid, which was introduced in \cite{CMS}.
	
	In Section \ref{sec6}, we prove that the Betti numbers satisfy $\beta_j(\HS_i(I^k))\le\beta_j(\HS_i(I^{k+1}))$ for all integers $i,j,k$ (Theorem \ref{Thm:bettiPol}), leading to the fact that the depth function $k\mapsto\depth\,S/\HS_i(I^k)$ is non-increasing. As an immediate consequence, if $\m=(x_1,\dots,x_n)\in\Ass\,\HS_i(I^k)$, then $\m\in\Ass\,\HS_i(I^{k+1})$ (Corollary \ref{Cor:Ass-m}). Based on this result, and experimental evidence, we expect that the sequence $\{\Ass\,\HS_i(I^k)\}_{k>0}$ becomes an increasing chain for $k\ge i$ (Conjecture \ref{Conj:AssPolym}). It turns out that Conjecture \ref{Conj:GenDeg} implies Conjecture \ref{Conj:AssPolym} (Theorem \ref{Thm:Implication}). Interestingly enough, to prove this result we employ the concept of $i$th \textit{homological strong persistence property} introduced in \cite{FQ1} and the so-called determinantal trick \cite[Corollary 1.1.8]{HS2006}. Hence, we establish Conjecture \ref{Conj:AssPolym} for $i=1$ and for the families (a), (b), (c) mentioned before.
	
	The product of polymatroidal ideals is polymatroidal. In Section \ref{sec8}, we prove that if $\m I$ is polymatroidal and $I$ has linear resolution, then $I$ is polymatroidal (Proposition \ref{Prop:mI=>IPolym}). This result was conjectured by Bandari and Herzog \cite{BH2013}.
	
	In the last two sections, we consider \textit{componentwise polymatroidal ideals}. These ideals, introduced in \cite{BH2013}, are monomial ideals whose all graded components $I_{\langle j\rangle}$ are polymatroidal ideals. It is proved in \cite{F-SCDP} that these ideals have linear quotients. In Section \ref{sec9}, we prove that the saturation of a (componentwise) polymatroidal ideal is componentwise polymatroidal (Corollary \ref{Cor:Isat}). We expect that $\HS_i(I)$ is componentwise polymatroidal for all $i$, if $I$ is such. We establish this expectation for $i=1$ (Theorem \ref{Thm:HS-cp}).
	
	In contrast to polymatroidal ideals, powers of componentwise polymatroidal ideals are no longer componentwise polymatroidal. However, it is expected that they are componentwise linear. Whether $\HS_i(\mathcal{R}(I))$ is a finitely generated $\mathcal{R}(I)$-module for a componentwise polymatroidal ideal $I\subset S$ is an open question at the moment.
	
	\section{Homological shift algebras}\label{sec2}
	
	Let $I\subset S$ be a graded ideal and let $\p\in\Ass\,I$ be an associated prime. Recall that the \textit{regularity} of $I$ is defined as $\reg\,I=\max\{j:\ \Tor^S_i(K,I)_{i+j}\ne0\}$, the \textit{$\v_\p$-number} of $I$ is defined as $\v_\p(I)=\min\{\deg(f):\ f\in S_d,\ (I:f)=\p\}$, and the \textit{$\v$-number} of $I$ is defined as $\v(I)=\min\limits_{\p\in\Ass\,I}\v_\p(I)$. See, also, \cite{Conca23,CSTVV,F2023,FS2,FSPack,FSPackA}.\smallskip
	
	Let $I\subset S$ be a monomial ideal and let $\mathcal{G}(I)$ be its minimal monomial generating set. The $i$th \textit{homological shift algebra} of $I$ is introduced in \cite{FQ} as the $K$-algebra,
	$$
	\HS_i(\mathcal{R}(I))\ =\ \bigoplus_{k\ge1}\HS_i(I^k).
	$$
	
	We say that $I\subset S$ has \textit{linear powers} if $I^k$ has linear resolution for all $k\ge1$. The next results were proved in \cite[Theorem 1.4]{FQ} and \cite[Theorem 2.1]{FQ}.
	\begin{Theorem}\label{Thm:resume-HS(R(I))}
		Let $I\subset S$ be a monomial ideal with linear powers. Then $\HS_i(\mathcal{R}(I))$ is a finitely generated graded $\mathcal{R}(I)$-module. Hence, $I\cdot\HS_i(I^k)\subset\HS_i(I^{k+1})$ for all $k\ge1$ and equality holds for all $k\gg0$. Furthermore, the following statements hold.
		\begin{enumerate}
			\item[\textup{(a)}] The set $\Ass\,\HS_i(I^k)$ stabilizes: $\Ass\,\HS_i(I^{k+1})=\Ass\,\HS_i(I^{k})$ for all $k\gg0$.
			\item[\textup{(b)}] For all $k\gg0$, we have $\depth\, S/\HS_i(I^{k+1})=\depth\,S/\HS_i(I^k)$.
			\item[\textup{(c)}] For all $k\gg0$, $\reg\,\HS_i(I^k)$ is a linear function in $k$.
			\item[\textup{(d)}] For all $k\gg0$, $\v(\HS_i(I^k))$ is a linear function in $k$.
			\item[\textup{(e)}] For all $k\gg0$ and all $\p\in\Ass\,\HS_i(I^k)$, $\v_\p(\HS_i(I^k))$ is a linear function in $k$.
		\end{enumerate}
	\end{Theorem}
	
	\section{Polymatroidal ideals}\label{sec3}
	
	A monomial ideal $I\subset S$ is called \textit{polymatroidal} if the exponent vectors of the minimal monomial generators of $I$ form the set of bases of a discrete polymatroid. A squarefree polymatroidal ideal is called \textit{matroidal}.
	
	For a monomial $u=x_1^{a_1}\cdots x_n^{a_n}\in S$, we define the \textit{$x_i$-degree} of $u$ as the integer $$\deg_{x_i}(u)=\max\{j:\ x_i^j\ \textup{divides}\ u\}=a_i.$$
	
	A monomial ideal $I\subset S$ is polymatroidal, if and only if, it is generated in one degree and the \textit{exchange property} holds: For all $u,v\in\mathcal{G}(I)$ with $\deg_{x_i}(u)>\deg_{x_i}(v)$ for some $i$, there exists $j$ with $\deg_{x_j}(u)<\deg_{x_j}(v)$ such that $x_j(u/x_i)\in\mathcal{G}(I)$.\smallskip
	
	By \cite[Theorem 12.4.1]{HHBook}, polymatroidal ideals also satisfy the so-called \textit{symmetric exchange property}: For all $u,v\in\mathcal{G}(I)$ and all $i$ such that $\deg_{x_i}(u)<\deg_{x_i}(v)$ there exists $j$ with $\deg_{x_j}(u)>\deg_{x_j}(v)$ such that $x_i(u/x_j)$, $x_j(v/x_i)\in\mathcal{G}(I)$.\smallskip
	
	Polymatroidal ideals have linear quotients with respect to the lexicographic order of $\mathcal{G}(I)$, induced by any ordering of the variables of $S$. Recall that $I$ has {\em linear quotients} if the set $\mathcal{G}(I)$ can be ordered as $u_1,\ldots,u_m$ such that for each $i = 2,\ldots,m$, the ideal $(u_1,\ldots,u_{i-1}):(u_i)$ is generated by variables. By \cite[Lemma 2.1]{JZ}, we can always assume that $\deg(u_1)\le\dots\le\deg(u_m)$. We put $\set(u_1)=\emptyset$ and
	$$
	\set(u_i)\ =\ \{\ell\ :\ x_\ell\in(u_1,\dots,u_{i-1}):(u_i)\},\quad\textup{for}\ i=2,\dots,m.
	$$
	
	By \cite[Lemma 1.5]{ET}, if $I\subset S$ has linear quotients, we have
	\begin{equation}\label{eq:HerTak}
		\HS_i(I)\ =\ ({\bf x}_Fu:\ u\in\mathcal{G}(I),\ F\subset\set(u),\ |F|=i),
	\end{equation}
	where ${\bf x}_F=\prod_{j\in F}x_j$ if $F\ne\emptyset$, and ${\bf x}_\emptyset=1$ otherwise.\smallskip
	
	Products of polymatroidal ideals are polymatroidal \cite[Theorem 12.6.3]{HHBook}. Therefore, a polymatroidal ideal $I\subset S$ has linear powers, and so by Theorem \ref{Thm:resume-HS(R(I))} the homological shift algebras $\HS_i(\mathcal{R}(I))$ are finitely generated $\mathcal{R}(I)$-modules.
	
	\section{The first homological shift algebra of a polymatroidal ideal}\label{sec4}
	
	In this section, our main goal is to describe the first homological shift algebra of a polymatroidal ideal.
	\begin{Theorem}\label{Thm:HS1-fg}
		Let $I\subset S$ be a polymatroidal ideal. Then, $$\HS_1(I^{k+1})=I\cdot\HS_1(I^k),$$ for all $k\ge1$. In particular, $\HS_1(\mathcal{R}(I))$ is generated in degree one as a $\mathcal{R}(I)$-module.
	\end{Theorem}
	This theorem will be an immediate consequence of the following Proposition \ref{Prop:HS1-polym}, which was shown by Bandari \cite{Ban24}. Due to its importance, we provide a proof of it.
	
	For a monomial ideal $I\subset S$, its \textit{bounding multidegree} is defined as the vector ${\bf deg}(I)=(\deg_{x_1}(I),\dots,\deg_{x_n}(I))$ where
	$$
	\deg_{x_i}(I)\ =\ \max_{u\in\mathcal{G}(I)}\deg_{x_i}(u),\quad\textup{for all}\ i=1,\dots,n.
	$$
	
	Given a monomial ideal $I\subset S$ and a vector ${\bf a}=(a_1,\dots,a_n)\in\ZZ_{\ge0}^n$, we define the \textit{restriction of $I$ at ${\bf a}$} as the monomial ideal $I^{\le{\bf a}}$ with minimal generating set
	$$
	\mathcal{G}(I^{\le{\bf a}})\ =\ \{u\in\mathcal{G}(I):\ \deg_{x_i}(u)\le a_i,\ \textup{for all}\ i=1,\dots,n\}.
	$$
	
	Clearly, $I=I^{\le{\bf a}}$ whenever ${\bf a}\ge{\bf deg}(I)$. If $I$ is polymatroidal, then $I^{\le{\bf a}}$ is polymatroidal for any ${\bf a}\in\ZZ_{\ge0}^n$, see \cite[Lemma 4.1]{FM}.\smallskip
	
	We denote by $\alpha(I)$ the smallest degree of a minimal homogeneous generator of $I$. By $\m=(x_1,\dots,x_n)$ we denote the graded maximal ideal of $S$.\smallskip
	
	\begin{Proposition}\label{Prop:HS1-polym}
		Let $I\subset S$ be a polymatroidal ideal. Then,
		$$
		\HS_1(I)\ =\ (\m I)^{\le{\bf deg}(I)}.
		$$
	\end{Proposition}
	\begin{proof}
		Let $\mathcal{G}(I)=\{u_1,\dots,u_m\}$. By \cite[Lemma 2.2]{FQ1}, we have
		$$
		\HS_1(I)=(\lcm(u_i,u_j):\ 1\le i<j\le m).
		$$
		
		Let $w\in\mathcal{G}(\HS_1(I))$. Since $I$ has linear resolution, $\HS_1(I)$ is generated in degree $\alpha(I)+1$. Hence $w=\lcm(u_i,u_j)=x_su_i$ for some $s$, and $w\in\m I$. Notice that $\deg_{x_p}(w)=\max\{\deg_{x_p}(u_i),\deg_{x_p}(u_j)\}\le\deg_{x_p}(I)$ for all $p$. Hence $w\in(\m I)^{\le{\bf deg}(I)}$.
		
		Conversely, let $w=x_su\in\mathcal{G}((\m I)^{\le{\bf deg}(I)})$, with $u\in\mathcal{G}(I)$. Firstly, notice that $\deg_{x_s}(u)=\deg_{x_s}(w)-1<\deg_{x_s}(I)$. We can find a monomial $v\in\mathcal{G}(I)$ with $\deg_{x_s}(v)=\deg_{x_s}(I)>\deg_{x_s}(u)$. Applying the symmetric exchange property, there exists $r$ with $\deg_{x_r}(u)>\deg_{x_r}(v)$ such that $u'=x_s(u/x_r)\in\mathcal{G}(I)$. Since $\lcm(u,u')=x_su$, using again \cite[Lemma 2.2]{FQ1} we see that $(\m I)^{\le{\bf deg}(I)}\subset\HS_1(I)$, as desired.
	\end{proof}
	
	It is clear that if $I\subset S$ is an equigenerated monomial ideal, then
	$$
	{\bf deg}(I^k)\ =\ k\,{\bf deg}(I)\ =\ (k\deg_{x_1}(I),\dots,k\deg_{x_n}(I)),
	$$
	for all $k\ge1$. We are now ready to prove Theorem \ref{Thm:HS1-fg}.
	
	\begin{proof}[Proof of Theorem \ref{Thm:HS1-fg}]
		Using Proposition \ref{Prop:HS1-polym}, we have
		\begin{align*}
			\HS_1(I^{k+1})\ &=\ (\m I^{k+1})^{\le{\bf deg}(I^{k+1})}\ = \ (\m I^{k+1})^{\le(k+1){\bf deg}(I)}\\
			&=\ I^{\le{\bf deg}(I)}(\m I^k)^{\le k{\bf deg}(I)}\\
			&=\ I\cdot\HS_1(I^k),
		\end{align*}
		for all $k\ge1$. In the above equation, only the equality
		$$
		(\m I^{k+1})^{\le(k+1){\bf deg}(I)}\ = \ I^{\le{\bf deg}(I)}(\m I^k)^{\le k{\bf deg}(I)}
		$$
		needs to be justified. That the right-hand is included in the left-hand side is clear. Conversely, let $w\in\mathcal{G}((\m I^{k+1})^{\le(k+1){\bf deg}(I)})$. Then $w=u_1\cdots u_{k+1}x_s$ with $u_j\in\mathcal{G}(I)$ and $\deg_{x_s}(w)\le(k+1)\deg_{x_s}(I)$. Since $\deg_{x_s}(u_j)\le\deg_{x_s}(I)$ for all $j=1,\dots,k+1$, we can find some integer $j$ such that $\deg_{x_s}(u_{j})<\deg_{x_s}(I)$. Up to relabeling, we may assume $j=k+1$. Setting $w'=u_1\cdots u_kx_s$, we see that $w'\in\mathcal{G}((\m I^k)^{\le k{\bf deg}(I)})$ and so $w=u_{k+1}w'\in I^{\le{\bf deg}(I)}(\m I^k)^{\le k{\bf deg}(I)}$.
	\end{proof}
	
	\begin{Corollary}
		Let $I\subset S$ be a polymatroidal ideal. Then
		$$
		\HS_1(\mathcal{R}(I))\ =\ ((\m I)^{\le{\bf deg}(I)}\mathcal{R}(I))(1),
		$$
		where ``\,$(1)$" denotes the shift by degree one.
	\end{Corollary}
	
	\section{Finite generation of $\HS_i(\mathcal{R}(I))$}\label{sec5}
	In this section, we focus our attention on the generating degrees of the homological shift algebras of a polymatroidal ideal. Due to Theorem \ref{Thm:HS1-fg} and several experimental evidence, we expect that
	\begin{Conjecture}\label{Conj:GenDeg}
		Let $I\subset S$ be a polymatroidal ideal. Then $\HS_i(\mathcal{R}(I))$ is generated in degrees $\le i$ as a $\mathcal{R}(I)$-module, for all $i>0$.
	\end{Conjecture}
	
	This conjecture in the case $i=n-1$ has already been posed by Chu, Herzog and Lu in \cite{CHL}, and solved for the family of pruned path lattice polymatroidal ideals.\smallskip
	
	Theorem \ref{Thm:HS1-fg} settles Conjecture \ref{Conj:GenDeg} in the case $i=1$.\smallskip
	
	At the moment, we do not have a general strategy to settle Conjecture \ref{Conj:GenDeg} for $i\ge2$. Therefore, in what follows, we establish this conjecture for some large families of polymatroidal ideals. These families are:
	\begin{enumerate}
		\item[\textup{(a)}] Principal Borel ideals.
		\item[\textup{(b)}] Polymatroidal ideals satisfying the strong exchange property.
		\item[\textup{(c)}] Matroidal edge ideals.
	\end{enumerate}
	
	\subsection{Principal Borel ideals} Let $u=x_{j_1}\cdots x_{j_d}\in S$ be a monomial of degree $d$, with $1\le j_1\le\dots\le j_d\le n$. The \textit{principal Borel ideal} generated by $u$ is the monomial ideal
	$$
	B(u)\ =\ (x_{p_1}\cdots x_{p_d}:\ 1\le p_1\le\dots\le p_d\le n,\ p_s\le j_s\ \textup{for all}\ s=1,\dots,d).
	$$
	
	This ideal is polymatroidal, and therefore it has linear quotients with respect to the lexicographic order of $\mathcal{G}(B(u))$. For a monomial $v\in S$, $v\ne1$, we put $\max(v)=\max\{j:x_j\ \textup{divides}\ v\}$. It is not difficult to see that for all $v\in\mathcal{G}(B(u))$ we have $\set(v)=[\max(v)-1]=\{1,2,\dots,\max(v)-1\}$.
	
	Now, we can establish Conjecture \ref{Conj:GenDeg} for principal Borel ideals.
	\begin{Proposition}\label{Prop:B(u)}
		Let $I=B(u)$. Then $\HS_i(I^{k+1})=I\cdot\HS_i(I^k)$ for all $k\ge1$.
	\end{Proposition}
	\begin{proof}
		It is immediate to see that $I^s=B(u)^s=B(u^s)$ for all $s\ge1$. Now, let $w\in\HS_i(I^{k+1})$. By the previous description and equation (\ref{eq:HerTak}) we can write $w={\bf x}_Fv$ with $v\in\mathcal{G}(I^{k+1})=\mathcal{G}(B(u^{k+1}))$ and $F\subset[\max(v)-1]$ is a subset of cardinality $i$. We can write $v=v_1\cdots v_{k+1}$ with each $v_s\in\mathcal{G}(I)$. We can assume that $\max(v)=\max(v_1)$. Then $v'=v_1\cdots v_{k}\in\mathcal{G}(I^k)=\mathcal{G}(B(u^k))$. Since $\set(v')=\set(v)=[\max(v)-1]$ we obtain that $w'={\bf x}_Fv'\in\HS_i(I^k)$. Hence $w=v_{k+1}w'\in I\cdot\HS_i(I^k)$. This shows that $\HS_i(I^{k+1})\subset I\cdot\HS_i(I^k)$ for all $k\ge1$. Since the converse inclusion holds as well for all $k\ge1$, the assertion follows.
	\end{proof}
	
	\subsection{Polymatroidal ideals satisfying the strong exchange property}
	We say that a monomial ideal $I\subset S$ generated in a single degree satisfy the \textit{strong exchange property} if: for all $u,v\in\mathcal{G}(I)$, all $i$ such that $\deg_{x_i}(u)>\deg_{x_i}(v)$ and all $j$ such that $\deg_{x_j}(u)<\deg_{x_j}(v)$ we have $x_j(u/x_i)\in I$.
	
	It is clear that ideals satisfying the strong exchange property, satisfy the exchange property as well, and therefore are polymatroidal.
	
	Given a vector ${\bf a}=(a_1,\dots,a_n)\in\ZZ_{\ge0}^n$ and an integer $d$ with $a_1+\dots+a_n\ge d$, the ideal of \textit{Veronese type} (determined by ${\bf a}$ and $d$) is the polymatroidal ideal $I_{{\bf a},d}$ with minimal monomial generating set
	$$
	\mathcal{G}(I_{{\bf a},d})\ =\ \{{\bf x^b}\in S:\ |{\bf b}|=d,\ {\bf b}\le{\bf a}\}.
	$$
	Here, if ${\bf b}=(b_1,\dots,b_n)\in\ZZ_{\ge0}^n$ we put ${\bf b}(i)=b_i$, $|{\bf b}|=\sum_{i=1}^n{\bf b}(i)$ is the modulus of ${\bf b}$ and ${\bf b}\le{\bf a}$ means that ${\bf b}(i)\le{\bf a}(i)$ for all $i$.
	
	Ideals of Veronese type have been studied in \cite{HHV,HRV}. By the work \cite{HHV}, it is known that ideals satisfying the strong exchange property are essentially ideals of Veronese type in the following sense.
	\begin{Proposition}\label{Prop:HHV}
		Let $I\subset S$ be a monomial ideal generated in a single degree. Then $I$ satisfies the strong exchange property, if and only if, $I=(u)I_{{\bf a},d}$ for some monomial $u\in S$, and some ${\bf a}\in\ZZ_{\ge0}^n$ and $d>0$.
	\end{Proposition}
	
	We are going to prove the following result.
	\begin{Proposition}\label{Prop:VerType}
		Let $I$ be a polymatroidal ideal satisfying the strong exchange property. Then $\HS_i(I^{k+1})=I\cdot\HS_i(I^k)$ for all $k\ge1$.
	\end{Proposition}
	
	The proof of this result requires some preparations.\smallskip
	
	We follow the notation from \cite{HHBook} and \cite{CMS}. Let $P$ be an integral polymatroid on $[n]$ with  rank function $\rho$. We use $B(P)$ to denote  the set of bases of $P$, that is,
	$$
	B(P)\ =\ \{{\bf u}\in P\cap \mathbb{N}^n\ :\ {\bf u}([n])=\rho([n])\}.
	$$
	Here, for $A\subset [n]$ and a vector ${\bf u}\in\mathbb{R}^n$, ${\bf u}(A)$ denotes the sum $\sum_{i\in A}{\bf u}(i)$.\smallskip
	
	A {\it cage} for $P$ is an integral vector ${\bf c}\in \mathbb{N}^n$ such that ${\bf u}(i)\le{\bf c}(i)$ for any ${\bf u}\in P$ and any $i\in [n]$. Let $I$ be the polymatroidal ideal of $P$. That is, the ideal $I$ generated by the monomials ${\bf x^u}$ with ${\bf u}\in B(P)$. Then ${\bf deg}(I)$ is a cage for $P$, and for any other cage ${\bf c}$ of $P$ we have ${\bf c}\ge{\bf deg}(I)$.
	
	Fix a cage ${\bf c}$ of $P$. The set $\{{\bf c-u}:\ {\bf u}\in B(P)\}$ is then the set of bases of a polymatroid, which we call the \textit{dual polymatroid} of $P$ with respect to ${\bf c}$ and denote by $P_{\bf c}^{\vee}$, or simply by $P^{\vee}$. See \cite[Theorem 1.3]{FM1} for a related result. The rank function of $P_{\bf c}^{\vee}$ is given by $$\rho^{\vee}(A)\ =\ \rho([n]\setminus A)-\rho([n])+{\bf c}(A).$$
	
	The M\"obius function of a polymatroid was introduced in \cite{CMS}. However, it has a close connection to the classical notion of a M\"obius function defined on a partially ordered set (poset), which we recall now. Let $Q$ be a finite poset. The incidence algebra $I(Q) $ of a poset $Q$ consists of all functions $f: Q\times Q \to \mathbb{Z} $ such that $ f(x, y) = 0 $ whenever $ x \not\leq y $.
	
	Convolution is the multiplication operation in $I(Q)$. For two functions $f, g \in I(Q) $, their convolution $ f * g $ is defined by:
	\[
	(f * g)(x, y)\ =\ \sum_{\substack{z \in Q \\ x \leq z \leq y}} f(x, z) \cdot g(z, y)
	\]
	where the sum ranges over all elements $ z $ that lie “between” $ x $ and $ y $ in the poset (i.e., $ z $ satisfies $ x \leq z \leq y $).
	
	Convolution is associative, meaning $ (f * g)*  h = f*  (g*  h) $, and the incidence algebra $ I(Q) $ becomes a ring under this operation (with addition defined pointwise).

	The zeta function, denoted by $ \zeta $, is a fundamental function in $I(Q)$. Formally, $\zeta: Q \times Q \to \mathbb{Z} $  is defined by:
	\[
	\zeta(x, y)\ =\
	\begin{cases}
		\,1 & \text{if } x \leq y, \\
		\,0 & \text{otherwise}.
	\end{cases}
	\]
	Intuitively, $ \zeta(x, y) $ “detects” whether $ x $ is less than or equal to $ y $ in the poset.

	The M\"obius function, denoted $\mu $, is the inverse of the zeta function with respect to convolution in the incidence algebra. Formally, $ \mu \in I(Q) $ satisfies:
	\[
	\zeta\mu\ =\ \mu\zeta\ =\ \delta,
	\]
	where $ \delta $ is the identity element of the incidence algebra, defined by:
	\[
	\delta(x, y)\ =\
	\begin{cases}
		\,1 & \text{if } x = y, \\
		\,0 & \text{otherwise}.
	\end{cases}
	\]
	
	The M\"obius function can be computed recursively using the definition of convolution. For $ x, y \in Q$:
	\begin{enumerate}
		\item[-] If $ x = y $, $ \mu(x, y) = 1 $.
		\item[-] If $ x < y $ (i.e., $ x \leq y $ but $ x \neq y $), then $$\mu(x, y)\ =\ -\sum_{\substack{z \in Q\\ x \leq z < y}} \mu(x, z)\ =\ -\sum_{\substack{z \in Q\\ x< z \leq y}} \mu(z, y).$$
		\item[-] If $ x \not\leq y $, $ \mu(x, y) = 0 $.
	\end{enumerate}\smallskip
	
	We now turn to the definition of the M\"obius function on a polymatroid $P$.
	\begin{Definition} \em Consider $P\cap\mathbb{N}^{n}$ as a finite poset, where ${\bf u}\le{\bf v}$ if and only if ${\bf u}(i)\le{\bf v}(i)$ for all $i\in[n]$. Adjoin the element $\widehat{1}$ to the poset $P\cap\mathbb{N}^{n}$ such that $\widehat{1}$ becomes the maximal element of the resulting poset. Let $\mu_{1}$ denote the classical  M\"obius function on this resulting set. The M\"obius function on the polymatroid $P$ is then defined as the function $\mu_P$ from $P\cap\mathbb{N}^{n}$ to $\mathbb{N}$ given by
		$$
		\mu_P({\bf u})=-\mu_1({\bf u}, \widehat{1} ), \quad \textup{for all}\ {\bf u}\in P\cap \mathbb{N}^{n}.
		$$
	\end{Definition}
	It is not difficult to see that this definition is equivalent to \cite[Definition 2.1]{CMS}.\smallskip
	
	Let ${\bf e}_1,\dots,{\bf e}_n$ be the standard basis of $\mathbb{R}^n$. That is, ${\bf e}_i(j)=0$ if $j\ne i$ and ${\bf e}_i(i)=1$.\smallskip
	
	Let $I$ be the polymatroidal ideal of the polymatroid $P$ with cage ${\bf c}$. Let ${\bf u}\in B(P)$ and $1\leq j_1<j_2<\cdots<j_i\leq n$. According to \cite[Theorem A(iii)]{CMS}, the monomial $x_{j_1}x_{j_1}\cdots x_{j_i}{\bf x^u}$ belongs to $\HS_i(I)$ if and only if $\mu_{P^{\vee}}({\bf c-u}-{\bf e}_{j_1}-\ldots-{\bf e}_{j_i})\neq 0$.
	
	Let $I$ be the polymatroidal ideal of the integral polymatroid $P$. Then $I^{k+1}$ is the polymatroidal ideal of $P_{k+1}=P\vee P\vee\cdots\vee P$, where $P$ repeats $(k+1)$ times, see \cite[Theorem 12.1.5]{HHBook}. Note that if ${\bf c}$ is a cage of $P$, then $(k+1){\bf c}$ is a cage of $P_{k+1}$. It follows that $P_{k+1}^{\vee}=(P^{\vee})_{k+1}$. Let $\mu_{k+1}^{\vee}$ denote the M\"obius function on $P_{k+1}^{\vee}$.
	
	For integers $d>0$ and $a_1,\ldots,a_n$ with $a_1+\cdots+a_n\ge d$, let
	$$
	B\ =\ \{{\bf u}\in \mathbb{N}^n\ :\ {\bf u}(i)\le a_i,\ \textup{for all}\ i=1,\dots,n,\ {\bf u}([n])=d\}.
	$$
	
	Then $B$ represents the set of bases of a discrete polymatroid $P$ on the ground set $[n]$, of rank $d$, which is called a discrete polymatroid of Veronese type. Its polymatroidal ideal $I\subset S$ is the ideal of Veronese type $I_{{\bf a},d}$ with ${\bf a}=(a_1,\dots,a_n)$.
	
	\begin{Lemma} \label{Mobius} Fix positive integers  $i, n$ and let ${\bf c}=(a_1,\ldots,a_n)\in \mathbb{N}_{>0}^n$. Let
		$$
		P=\{(y_1,\ldots,y_n)\in \mathbb{R}^n:\ y_1+\cdots+y_n\leq i,\ 0\leq y_s\leq a_s,\ \text{for all}\ s=1,\ldots,n\}.
		$$
		Let $\mu $ be the M\"obius function on $P$. Then $\mu(\mathbf{0})=0$ if and only if $n\leq i$.
	\end{Lemma}
	
	\begin{proof}
		It is clear that ${\bf c}$ is a cage for $P$. Let $I$ be the monomial ideal generated by ${\bf x^{c-u}}$ with ${\bf u}\in B(P)$. In other words, $I$ is the polymatroidal ideal of the dual polymatroid of $P$ with respect to ${\bf c}$. If $i\le n$, let $\boldsymbol{\alpha}$ be the vector $(1,\ldots, 1,0,\ldots,0)$, where the first $i$ entries are 1's , and the remaining entries are 0's. It is clear that $\boldsymbol{\alpha}\in B(P)$. Otherwise, if $i>n$, let $\boldsymbol{\alpha}$ be any vector $(\alpha_1,\ldots,\alpha_n)\in B(P)$ with $\alpha_j>0$ for all $j=1,\ldots,n$. Note that, if $i\leq n$ we may write $\mathbf{0}=\boldsymbol{\alpha}-{\bf e}_1-\cdots-{\bf e}_i$.
		According to \cite[Theorem A(iii)]{CMS}, we have $\mu(\mathbf{0})\neq 0$ if and only if ${\bf x}^{{\bf c}-\boldsymbol{\alpha}}{\bf x}^{\boldsymbol{\alpha}}$ (which is equal to $x_1\cdots x_i{\bf x}^{{\bf c}-\boldsymbol{\alpha}}$ if $i\leq n$) belongs to $\HS_i(I)$.\smallskip
		
		(a) If $i\ge n$, since the projective dimension of $I$ is at most $n-1$, we have $\HS_i(I)=0$, which implies $\mu(\mathbf{0})= 0$.\smallskip
		
		(b) If $i<n$, we set $\boldsymbol{\alpha}_s={\bf e}_1+\cdots+\widehat{{{\bf e}_s}}+\cdots+{\bf e}_{i+1}$, for $s=1,\ldots,i$. Here, $\widehat{{{\bf e}_s}}$ means that ${\bf e}_s$ is omitted from the sum. Then all $\boldsymbol{\alpha}_s$ belong to $B(P)$. Set $u={\bf x}^{{\bf c}-\boldsymbol{\alpha}}=x_1^{a_1-1}x_2^{a_2-1}\cdots x_i^{a_i-1}x_{i+1}^{a_{i+1}}\cdots x_n^{a_n}$. It follows that ${\bf x}^{{\bf c}-\boldsymbol{\alpha}_s}=(x_su)/x_{i+1}$ are minimal generators of $I$, for all $s=1,\dots,i$. Moreover
		$$
		(x_1u)/x_{i+1}>(x_2u)/x_{i+1}>\cdots>(x_iu)/x_{i+1}>u
		$$
		with respect to the lexicographic order induced by $x_1>\dots>x_{n}$. Since $I$ has linear quotients with respect to any lexicographic order, it follows that $[i]\subset\set(u)$. Equation (\ref{eq:HerTak}) implies that $x_1x_2\cdots x_i u\in\HS_i(I)$ and so $\mu(\mathbf{0})\neq 0$, as desired.
	\end{proof}
	
	We are finally ready to prove Proposition \ref{Prop:VerType}.
	
	\begin{proof}[Proof of Proposition \ref{Prop:VerType}]
		Notice that if $J\subset S$ is a monomial ideal and $u\in S$ is a monomial, then $\HS_i(uJ)=(u)\HS_i(J)$. Therefore, in view of Proposition \ref{Prop:HHV}, it is enough to prove the statement for polymatroidal ideals of Veronese type.
		
		Assume that $I$ is the polymatroidal ideal of $P$, where
		$$
		P=\{(y_1,\ldots,y_n)\in \mathbb{R}^n:\ y_1+\cdots+y_n\leq d,\ 0\leq y_i\leq a_j,\ \textup{for all}\ j=1,\ldots,n\}.
		$$
		Here, ${\bf c}=(a_1,\ldots,a_n)$ is a given vector in $\mathbb{N}^n$.  When discussing the dual polymatroid of $P_k$ or $P_{k+1}$, we take  $k({\bf c+1})=k(a_1+1,\ldots,a_n+1)$ or $(k+1)({\bf c+1})$  as the cages respectively, so as to avoid unnecessary confusion.
		
		Fix $1\leq i\leq n-1$ and $k\geq 1$. Let ${\bf x}^{\boldsymbol{\alpha}}$ be a minimal generator of  $\HS_i(I^{k+1})$. Then we may write $\boldsymbol{\alpha}={\bf u}+{\bf e}_{j_1}+\cdots+{\bf e}_{j_i}$ for some vector ${\bf u}\in B(P_{k+1})$ and some $1\leq j_1<j_2<\cdots<j_i\leq n$. Let $P_{k+1}^{\vee}$ be the dual polymatroid of $P_{k+1}$ with respect to the cage $(k+1)({\bf c+1})$. Then $P_{k+1}^{\vee}$ is the following polymatroid:
		$$
		\Big\{(y_1,\ldots,y_n)\in\mathbb{R}^n\ \Big\vert\ \ \substack{\displaystyle y_1+\cdots+y_n\leq (k+1)\big({\bf c}([n])+n-d\big)\\[3pt] \displaystyle 0\leq y_j\leq (k+1)a_j,\ 1\leq j\leq n}\Big\}.
		$$
		Let $\mu_{k+1}^{\vee}$ be the M\"obius function on $P_{k+1}^{\vee}$. According to \cite[Theorem A(iii)]{CMS}, we have $\mu_{k+1}^{\vee}((k+1)({\bf c+1}) -\boldsymbol{\alpha})\neq 0$. Note that, the poset $$\{{\bf v}\in P_{k+1}^{\vee}\cap\mathbb{N}^n\ :\ {\bf v}\geq (k+1)({\bf c+1})-\boldsymbol{\alpha}\}$$ is isomorphic to the poset $$\{(y_1,\ldots,y_n)\in \mathbb{N}^n\ :\  y_1+\cdots+y_n\leq i,\ 0\leq y_j\leq {\bf u}(j)+k_j\}.$$  Here, $k_j=1 $ if $j\in \{j_1,\ldots,j_i\}$ and 0 otherwise. Note that this poset is the intersection of $\mathbb{N}^n$ with a polymatroid, which we denote by $Q$. Since   $\mu_{k+1}^{\vee}((k+1)({\bf c+1}) -\boldsymbol{\alpha})\neq 0$, we have $\mu_Q({\bf 0})\neq 0$. By Lemma~\ref{Mobius}, it follows that there exists $j\in [n]\setminus\{j_1,\ldots,j_i\}$ such that ${\bf u}(j)\neq 0$.
		
		Write ${\bf u}={\bf u}_1+\cdots+{\bf u}_{k+1}$, with each ${\bf u}_i\in B(P)$. Without loss of generality, we assume that ${\bf u}_1(j)\neq 0$. Let $\boldsymbol{\beta}={\bf u}_1+\ldots+{\bf u}_k+{\bf e}_{j_1}+\cdots+{\bf e}_{j_i}$. Let $P_k^{\vee}$ be the dual polymatroid of $P_k$ with respect to $k(\mathbf{c+1})$ and let $\mu_k^{\vee}$ be its M\"obius function. Then $\{{\bf v}\in P_k^{\vee}\ :\ {\bf v}\geq k({\bf c+1})-\boldsymbol{\beta}\}$ is isomorphic to the poset $$\{(y_1,\ldots,y_n)\in \mathbb{N}^n\ :\  y_1+\cdots+y_n\leq i,\ 0\leq y_j\leq {\bf u}_1(j)+\cdots+{\bf u}_k(j)+k_j\}.$$ Here, $k_j=1$ if $j\in \{j_1,\ldots,j_i\}$ and $0$ otherwise.
		
		Since ${\bf u}_1(j)\neq 0$, it follows that $\mu_k^{\vee}(k({\bf c+1})-\boldsymbol{\beta})\neq 0$ by Lemma~\ref{Mobius}. This implies ${\bf x}^{\boldsymbol{\beta}}\in\HS_i(I^{k+1})$ and so ${\bf x}^{\boldsymbol{\alpha}}={\bf x}^{{\bf u}_{k+1}}{\bf x}^{\boldsymbol{\beta}}$ belongs to $I\cdot\HS_i(I^{k})$, as desired.
	\end{proof}
	
	One may conjecture that if $I=I_{{\bf a},d}$, then $\HS_i(I)=I_{{\bf a},d+i}$ for $i=1, \ldots, n-1$. This equality holds for $i=1$. Furthermore, $\HS_i(I)\subset I_{{\bf a},d+i}$ for all $i \geq 1$. However, the converse inclusion does not hold in general, as shown next.
	
	\begin{Example}\em 
		Let ${\bf a}=(10,10,10,10)$, $d=14$ and $I=I_{{\bf a},d}$. Then $x_1^{10}x_2^7$ belongs to $I_{{\bf a},17}$. However, every monomial in the minimal generating set of $\HS_3(I) $ is of the form $x_{j_1}x_{j_2}x_{j_3}u$, where $u\in\mathcal{G}(I)$ and $1 \leq j_1 < j_2 < j_3 \leq 4$. Hence $x_1^{10}x_2^7\notin\HS_3(I)$.
	\end{Example}
	This example also shows that the homological shift ideals of a Veronese type ideal are not of Veronese type in general.
	
	\subsection{Matroidal edge ideals} For a finite simple graph $G$ on the vertex set $[n]$ and with edge set $E(G)$, the \textit{edge ideal} of $G$ is defined as the squarefree monomial ideal $$I(G)=(x_ix_j:\ \{i,j\}\in E(G)).$$
	
	\begin{Definition}\label{Def:multi}
		\rm Given positive integers $a \leq b$, we set $[a,b]=\{c\in\mathbb{Z}: a\leq c\leq b\}$. Let $G$ be a graph with vertex set $V(G) = [n]$ and a vertex partition $V_1, \ldots, V_r$, where
		$$
		V_i=[t_{i-1}+1, t_i],\quad \text{ for all }\,\,\, i=1 , \ldots, r,
		$$
		with $t_0 = 0$, $t_{0}<t_1<\cdots<t_{r-1}$ and $t_r = n$. If $E(G) = \{ \{i, j\} :\ i \in V_k,\ j \in V_\ell \text{ and } k \neq \ell\}$, then $G$ is called a \textit{complete multipartite graph}.
		
	\end{Definition}
	The following result is proved in \cite[Theorem 3.1]{FM1}, see also \cite{KNQ}. 
	
	\begin{Proposition}\label{Prop:I(G)matr}
		The ideal $I(G)$ is matroidal if and only if $G$ is a complete multipartite graph.
	\end{Proposition}
	
	Next, we establish Conjecture \ref{Conj:GenDeg} for matroidal edge ideals.
	
	\begin{Proposition}\label{Prop:I(G)}
		Let $I=I(G)\subset S$ be a matroidal edge ideal. Then $\HS_i(I^{k+1})=I\cdot\HS_i(I^k)$ for all $k\ge i$.
	\end{Proposition}
	
	We need the following lemma shown in \cite[Proposition 4.7]{FQ1}.
	
	Let $u=x_1^{a_1}\cdots x_n^{a_n}\in S$ be a monomial and $A\subset[n]$ be non-empty. We put $u_A=\prod_{i\in A}x_i^{a_i}$ and define the support of $u$ as $\supp(u)=\{i:\ a_i>0\}$.
	\begin{Lemma}\label{Lem:Ay}
		Let $G$ be a complete multipartite graph as given in Definition \ref{Def:multi}, and put $I=I(G)$. Then, $I^k$ has linear quotients with respect to the lexicographical order induced by $x_1>\cdots>x_n$. Let $u\in\mathcal{G}(I^k)$ with $\max(u)\in V_d$ for some $d$. Then, the following statements hold.
		\begin{enumerate}
			\item[\textup{(i)}] If there exists some $d'<d$ such that $\deg u_{V_{d'}} = k$, then 
			\[
			\set(u)=[1,p-1] \cup [{t_{d'}+1}, {\max(u)-1}]
			\]
			where $p$ is the maximal integer with $p\in\supp(u_{V_{d'}})$.\smallskip
			\item[\textup{(ii)}] If $\deg u_{V_{d'}}< k$ for all $d'< d$, then $\set(u)=[1, \max(u)-1]$.
		\end{enumerate}
	\end{Lemma}
	
	We are ready to prove Proposition \ref{Prop:I(G)}.
	\begin{proof}[Proof of Proposition \ref{Prop:I(G)}]
		By Proposition \ref{Prop:I(G)matr} we can write $I=I(G)$ where $G$ is a complete multipartite graph as given in Definition \ref{Def:multi}. For $i=0,1$ the statements hold by Theorem \ref{Thm:HS1-fg}. Let $i\ge2$. We will show that $\HS_i(I^{k+1})=I\cdot\HS_1(I^k)$ for all $k\ge2$. It is enough to show that $\HS_i(I^{k+1})\subset I\cdot\HS_1(I^k)$ for all $k\ge2$. Let $w\in\mathcal{G}(\HS_i(I^{k+1}))$ for some $k\ge2$. By equation (\ref{eq:HerTak}) we can write $w={\bf x}_F e_1\cdots e_{k+1}$ with $e_i\in\mathcal{G}(I)$ and $F$ is a subset of size $i$ of $\set(u)$, where $u=e_1\cdots e_{k+1}\in\mathcal{G}(I^{k+1})$. Let $\max(u)\in V_d$ for some integer $d$. According to Lemma \ref{Lem:Ay}, one of the two conditions (i)-(ii) holds.\smallskip
		
		Assume that (i) holds. Then $\deg u_{V_{d'}}=k+1$ and so $\deg (e_j)_{V_{d'}}=1$ for all $j=1,\dots,k+1$. Let $p$ be the maximal integer belonging to $\supp(u_{V_{d'}})$. Since $x_p$ divides $u$, it divides some $e_j$, say $e_1$. Now $x_{\max(u)}$ divides some $e_j$ as well. Up to relabeling, $x_{\max(u)}$ divides either $e_1$ or $e_2$. Let $u'=e_1\cdots e_{k}\in\mathcal{G}(I^k)$. Since $k\ge2$, we have $\max(u')=\max(u)\in V_d$, $\deg(u')_{V_{d'}}=k$ and $p\in\supp(u')$ is the maximal element belonging to $\supp((u')_{V_{d'}})$. Applying again Lemma \ref{Lem:Ay}(i), this then shows that $\set(u')=\set(u)$ and so $w'={\bf x}_Fu'\in\HS_i(I^k)$. Finally, $w=e_{k+1}w'\in I\cdot\HS_i(I^k)$.\smallskip
		
		Assume that (ii) holds. Then $\set(u)=[\max(u)-1]$. Suppose that $x_{\max(u)}$ divides $e_{k+1}=x_px_q$. We claim that there exists  $j\in\{1,\dots,k\}$ such that $\set(u/e_j)=\set(u)$. If this is the case, then $F\subset\set(u/e_j)$ and so $w=e_{j}(w/e_{j})\in I\cdot\HS_i(I^k)$, as desired. Suppose now that this is not the case. Then, $\set(u/e_1)\ne\set(u)$ and $\set(u/e_2)\ne\set(u)$. Since $k\ge2$ and $x_{\max(u)}$ divides $e_{k+1}$, we have
		$$
		\max(u/e_1)=\max(u/e_2)=\max(u).
		$$
		By Lemma \ref{Lem:Ay}(i), there exist $d_1,d_2<d$ such that $\deg(u/e_1)_{V_{d_1}}=\deg(u/e_2)_{V_{d_2}}=k$. Notice that $d_1\ne d_2$, otherwise $\deg(u)_{V_{d_1}}=k+1$, which by Lemma \ref{Lem:Ay}(i) would mean that $\set(u)\ne[\max(u)-1]$ against the assumption. Hence $d_1\ne d_2$. Since $e_{k+1}=x_px_q$, this implies, up to relabeling, that $p\in V_{d_1}$ and $q\in V_{d_2}$. This contradicts the fact that $\max(u)=\max(p,q)\in V_d$. The proof is complete.
	\end{proof}
	
	\section{Asymptotic Betti numbers of $\HS_i(I^k)$}\label{sec6}
	
	In this short section, given a polymatroidal ideal $I\subset S$, we analyze the asymptotic behavior of the homological invariants of $\HS_i(I^k)$. As an immediate consequence of Theorem \ref{Thm:CMS}, we have
	\begin{Corollary}
		Let $I\subset S$ be a polymatroidal ideal. Then
		$$
		\reg\,\HS_i(I^k)\ =\ \v(\HS_i(I^k))+1\ =\ \alpha(I)k+i,
		$$
		whenever $\HS_i(I^k)\ne(0)$.
	\end{Corollary}
	\begin{proof}
		Since $I^k$ is polymatroidal, Theorem \ref{Thm:CMS} implies that $\HS_i(I^k)$ is again polymatroidal. Suppose that $\HS_i(I^k)\ne(0)$. Then $\reg\,\HS_i(I^k)=\alpha(\HS_i(I^k))=\alpha(I)k+i$ and by \cite[Theorem 5.5]{F2023} we have $\v(\HS_i(I^k))=\alpha(\HS_i(I^k))-1$.
	\end{proof}
	
	In the next result, we prove that the graded Betti numbers of $\HS_i(I^k)$ increase as $k$ increases.
	
	\begin{Theorem}\label{Thm:bettiPol}
		Let $I\subset S$ be a polymatroidal ideal. Then
		$$
		\beta_j(\HS_i(I^k))\ \le\ \beta_j(\HS_i(I^{k+1}))
		$$
		for all $k\ge1$. In particular, the function $k\mapsto\depth\,S/\HS_i(I^k)$ is non-increasing. That is,
		$$
		\depth\,S/\HS_i(I^k)\ \ge\ \depth\,S/\HS_i(I^{k+1}),
		$$
		for all $k\ge1$.
	\end{Theorem}
	
	For the proof of this result, we need the following lemma.
	\begin{Lemma}\label{Lem:Pol1}
		Let $I,J\subset S$ be ideals with $d$-linear resolution such that $J\subset I$. Then
		$$
		\beta_j(J)\le\beta_j(I),\quad\textit{for all}\ j\ge0.
		$$
	\end{Lemma}
	\begin{proof}
		The short exact sequence $0\rightarrow J\rightarrow I\rightarrow I/J\rightarrow 0$ induces the exact sequence
		$$
		\Tor^S_{j+1}(K,I/J)_{j+d}\rightarrow\Tor^S_{j}(K,J)_{j+d}\rightarrow\Tor^S_{j}(K,I)_{j+d}.
		$$
		Since $I/J$ is generated in degree $d$, we have $\Tor^S_{j+1}(K,I/J)_{j+d}=0$. Therefore, the map $\Tor^S_{j}(K,J)_{j+d}\rightarrow\Tor^S_{j}(K,I)_{j+d}$ is injective. Since $I$ and $J$ have $d$-linear resolution, we deduce that $\beta_j(J)=\beta_{j,j+d}(J)\le\beta_{j,j+d}(I)=\beta_j(I)$.
	\end{proof}
	
	We are now ready to prove Theorem \ref{Thm:bettiPol}.
	\begin{proof}[Proof of Theorem \ref{Thm:bettiPol}]
		If $\HS_i(I^k)=(0)$ there is nothing to prove. Let $\HS_i(I^k)\ne(0)$.
		Since $I$ has linear powers, Theorem \ref{Thm:resume-HS(R(I))} guarantees that $I\cdot\HS_i(I^k)\subset\HS_i(I^{k+1})$. Now, let $u\in\mathcal{G}(I)$. Then $(u)\HS_i(I^k)\subset I\cdot\HS_i(I^k)\subset\HS_i(I^{k+1})$. By Lemma \ref{Lem:Pol1} and Theorem \ref{Thm:CMS} we have
		$$
		\beta_j(\HS_i(I^k))\ =\ \beta_j((u)\HS_i(I^k))\ \le\ \beta_j(\HS_i(I^{k+1})),
		$$
		for all $j$ and all $k\ge1$.
	\end{proof}
	
	Let $\m=(x_1,\dots,x_n)$ be the unique graded maximal ideal of $S$. As an immediate consequence of Theorem \ref{Thm:bettiPol} we have
	
	\begin{Corollary}\label{Cor:Ass-m}
		Let $I\subset S$ be a polymatroidal ideal. If $\m\in\Ass\,\HS_i(I^{k_0})$ for some $k_0>0$, then $\m\in\Ass\,\HS_i(I^k)$ for all $k\ge k_0$.
	\end{Corollary}
	\begin{proof}
		We have $\m\in\Ass\,\HS_i(I^{k_0})$ if and only if $\depth\,S/\HS_i(I^{k_0})=0$. Suppose that $\m\in\Ass\,\HS_i(I^{k_0})$ for some $k_0>0$. Since by Theorem \ref{Thm:bettiPol} the depth function $k\mapsto\depth\,S/\HS_i(I^k)$ is non-increasing, it follows that $\depth\,S/\HS_i(I^k)=0$ for all $k\ge k_0$. Hence $\m\in\Ass\,\HS_i(I^k)$ for all $k\ge k_0$.
	\end{proof}
	
	Corollary \ref{Cor:Ass-m} is not valid for an arbitrary monomial prime ideal.
	\begin{Example}
		\rm Let $I=(x_{1}x_{3},x_{1}x_{4},x_{1}x_{5},x_{2}x_{3},x_{2}x_{4},x_{2}x_{5},x_{3}x_{4},x_{3}x_{5})$. Then $I$ is polymatroidal, $P=(x_1,x_5)\in\Ass\,\HS_2(I)$ but $P\notin\Ass\,\HS_2(I^2)$. However, using the \textit{Macaulay2} \cite{GDS} package \textup{\texttt{HomologicalShiftIdeals}} \cite{FPack1} we have
		$$
		\Ass\,\HS_2(I^2)\subset\Ass\,\HS_2(I^3)\subset\Ass\,\HS_2(I^4)\subset\cdots.
		$$
		This fact is proven for any matroidal edge ideal in Corollary \ref{Cor:someCases}(c).
	\end{Example}
	
	\section{Asymptotic associated primes of $\HS_i(I^k)$}\label{sec7}
	
	We now turn our attention to the behavior of the sequence $\{\Ass\,\HS_i(I^k)\}_{k>0}$ when $I\subset S$ is a polymatroidal ideal. In Corollary \ref{Cor:Ass-m} we have seen that if $I\subset S$ is polymatroidal and $\m\in\Ass\,\HS_i(I^{k_0})$ for some $k_0>0$, then $\m\in\Ass\,\HS_i(I^k)$ for all $k\ge k_0$. Due to this result and several experimental evidence, we expect that
	\begin{Conjecture}\label{Conj:AssPolym}
		Let $I\subset S$ be a polymatroidal ideal. Then
		$$
		\Ass\,\HS_i(I^{i})\subset\Ass\,\HS_i(I^{i+1})\subset\Ass\,\HS_i(I^{i+2})\subset\cdots,\quad\text{for all}\ i>0.
		$$
	\end{Conjecture}
	
	When a monomial ideal $I\subset S$ satisfies the chain of inclusions
	$$
	\Ass\,\HS_i(I)\subset\Ass\,\HS_i(I^2)\subset\Ass\,\HS_i(I^3)\subset\cdots,
	$$
	for some $i$, we say that $I$ satisfies the $i$th \textit{homological persistence property}.
	
	Whereas, if
	$$
	\HS_i(I^{k+1}):I\ =\ \HS_i(I^k),\quad\textup{for all}\ k\ge1,
	$$
	for some $i$, we say that $I$ satisfies the $i$th \textit{homological strong persistence property}.
	
	In \cite[Proposition 3.1]{FQ1}, it is shown that if $\HS_i(I^{k+1}):I=\HS_i(I^k)$ for some $k$, then $\Ass\,\HS_i(I^k)\subset\Ass\,\HS_i(I^{k+1})$. Hence, the $i$th homological strong persistence property implies the $i$th homological persistence property.\medskip
	
	We need the following general result, which follows from the determinantal trick.
	\begin{Theorem}\label{Thm:IntCl-General}
		Let $I,J$ be ideals of a Noetherian domain $R$. Suppose that $I^{k}J$ is integrally closed for some $k\ge0$. Then $I^{k+1}J:I=I^{k}J$.
	\end{Theorem}
	\begin{proof}
		It is clear that $I^{k}J\subset I^{k+1}J:I$. Conversely, let $f\in I^{k+1}J:I$, then $fI\subset I^{k+1}J=(I^{k}J)I$. Using the determinantal trick \cite[Corollary 1.1.8]{HS2006}, and since $R$ is a domain, we see that $f$ is an integral element over $I^{k}J$. But $I^{k}J$ is integrally closed by assumption. We conclude that $f\in I^{k}J$, as desired.
	\end{proof}
	
	As a consequence of this result, we have
	
	\begin{Theorem}\label{Thm:Implication}
		Conjecture \ref{Conj:GenDeg} implies Conjecture \ref{Conj:AssPolym}.
	\end{Theorem}
	\begin{proof}
		Assume that Conjecture \ref{Conj:GenDeg} holds true and set $J=\HS_i(I^i)$ for some $i>0$. Since $\HS_i(\mathcal{R}(I))$ is generated in degrees $\le i$ as a $\mathcal{R}(I)$-module, $$\HS_i(I^{s})\ =\ I^{s-i}\cdot\HS_i(I^i)\ =\ I^{s-i}J,$$ for all $s\ge i$. In view of \cite[Proposition 3.1]{FQ1} mentioned before, it is enough to show that $I^{k+1}J:I=I^{k}J$ for all $k\ge0$. Since $S$ is a Noetherian domain and $I^{k}J=\HS_1(I^{k+i})$ is a polymatroidal ideal by Theorem \ref{Thm:CMS} and hence is integrally closed (see \cite[Theorem 2.4]{HRV}), Theorem \ref{Thm:IntCl-General} yields the conclusion.
	\end{proof}
	
	As an immediate consequence of this result and Theorem \ref{Thm:HS1-fg}, we have
	
	\begin{Corollary}
		Let $I\subset S$ be a polymatroidal ideal. Then $I$ satisfies the $1$st homological strong persistence property. That is, $\HS_1(I^{k+1}):I=\HS_1(I^k)$ for all $k\ge1$. In particular,
		$$
		\Ass\,\HS_1(I)\subset\Ass\,\HS_1(I^2)\subset\Ass\,\HS_1(I^3)\subset\cdots.
		$$
	\end{Corollary}
	
	Combining Theorem \ref{Thm:Implication} with Propositions \ref{Prop:B(u)}, \ref{Prop:VerType} and \ref{Prop:I(G)} we obtain
	\begin{Corollary}\label{Cor:someCases}
		Conjecture \ref{Conj:AssPolym} holds true for the following families of ideals.
		\begin{enumerate}
			\item[\textup{(a)}] Principal Borel ideals.
			\item[\textup{(b)}] Polymatroidal ideals satisfying the strong exchange property.
			\item[\textup{(c)}] Matroidal edge ideals.
		\end{enumerate}
	\end{Corollary}
	
	\section{Multiplicative properties of polymatroidal ideals}\label{sec8}
	
	Let $I\subset S$ be a monomial ideal. The \textit{socle} of $I$ is the ideal defined as
	$$
	\textup{soc}(I)\ =\ ({\bf x^a}:\ {\bf x^a}\in(I:\m)\setminus I).
	$$
	
	The following result is well-known, see \cite[Proposition 1.13]{HMRZa}.
	
	\begin{Proposition}\label{Prop:HSn-1}
		Let $I\subset S$ be a monomial ideal with linear resolution. Then $$\HS_{n-1}(I)\ =\ x_1x_2\cdots x_n\cdot\soc(I).$$
	\end{Proposition}
	
	The product of polymatroidal ideals is polymatroidal. In the next result, we show that a similar but opposite property holds when we multiply by $\m$.
	\begin{Proposition}\label{Prop:mI=>IPolym}
		Let $I\subset S$ be an ideal with linear resolution. Suppose that $\m I$ is polymatroidal. Then $\textup{soc}(\m I)=I$, and so $I$ is polymatroidal.
	\end{Proposition}
	\begin{proof}
		Let $\alpha(I)=d$ be the initial degree of $I$. Firstly, we show that $(\m I:\m)=I$. The inclusion $I\subset(\m I:\m)$ is clear. Suppose that there exists $f\in (\m I:\m)\setminus I$. Then $\deg(f)=d-1$ by \cite[Proposition 1.4]{CHL}. Since $\m I$ is generated in degree $d+1$, then $(\m I:\m)$ is generated in degrees $\ge d$. This a contradiction, since $\deg(f)=d-1$. Thus $(\m I:\m)=I$. Since any $u\in\mathcal{G}(I)$ does not belong to $\m I$ by degree reasons, it follows that $\textup{soc}(\m I)=I$. Finally, since $\m I$ is polymatroidal, Proposition \ref{Prop:HSn-1} and Theorem \ref{Thm:CMS} imply that $\textup{soc}(\m I)=I$ is polymatroidal too.
	\end{proof}
	
	Given monomial ideals $L,P\subset S$ with $P$ prime, the \textit{monomial localization} of $L$ at $P$ is defined as the monomial ideal $L(P)$ in $S(P)=K[x_{i}:x_i\in P]$, obtained by applying the substitutions $x_i\mapsto 1$ for all $x_i\notin P$.
	
	Proposition \ref{Prop:mI=>IPolym} was conjectured in \cite{BH2013}. In such a paper, Bandari and Herzog conjectured that a monomial ideal $I\subset S$ is polymatroidal, if and only if, $I(P)$ has a linear resolution, for all monomial prime ideals $P\subset S$ \cite[Conjecture 2.9]{BH2013}.
	
	The next result was proved in \cite[Proposition 2.8]{BH2013} under the additional assumption that $I$ does not have embedded associated primes. See also \cite[Proposition 1.13]{MN21} for a different proof.
	\begin{Corollary}\label{Cor:BH}
		Let $I\subset S$ be a monomial ideal with $\textup{height}\,I=n-1$ such that $I(P)$ has a linear resolution for all monomial prime ideals $P\subset S$. Then $I$ is polymatroidal.
	\end{Corollary}
	\begin{proof}
		Write $I=J\cap Q$, where $J$ is the intersection of the primary components of $I$ having height $n-1$ and $Q$ is a $\m$-primary ideal. Then any $P\in\Ass\,J$ is a minimal prime of $J$. Thus $S(P)/J(P)$ has finite length for all $P\in\Ass\,J$. Since $J(P)=I(P)$ has a linear resolution, \cite[Lemma 2.2]{BH2013} implies that $J(P)=P^k$ for some $k$. Hence, $$J\ =\ P_1^{k_1}\cap\dots\cap P_r^{k_r}$$ for some monomial prime ideals $P_j$ of height $n-1$. Using that $P_i+P_j=\m$ for all $i\ne j$, then \cite[Theorem 3.1]{FVT2007} guarantees that $J$ is componentwise polymatroidal. Let $\alpha(I)=d$. Then, $\m^{j-d}I=I_{\langle j\rangle}=J_{\langle j\rangle}$ for $j\gg0$, because $Q$ is $\m$-primary. Since $J$ is componentwise polymatroidal, $\m^{j-d}I$ is polymatroidal. Repeated applications of Proposition \ref{Prop:mI=>IPolym} yield the conclusion.
	\end{proof}
	
	\section{Componentwise polymatroidal ideals}\label{sec9}
	
	Let $I\subset S$ be a homogeneous ideal. Then $I=\bigoplus_{j\ge0}I_j$, where $I_j$ is the $K$-vector space spanned by the homogeneous polynomials of degree $j$ belonging to $I$. We denote by $I_{\langle j\rangle}$ the homogeneous ideal of $S$ generated by $I_{j}$.
	
	We say that a monomial ideal $I\subset S$ is \textit{componentwise polymatroidal} if $I_{\langle j\rangle}$ is a polymatroidal ideal for all $j$. If $I$ is a polymatroidal generated in degree $d$, then $I_{\langle j\rangle}=(0)$ for $j<d$ and $I_{\langle j\rangle}=\m^{j-d}I$ for $j\ge d$. Hence, polymatroidal ideals are componentwise polymatroidal. Componentwise polymatroidal ideals have linear quotients \cite{F-SCDP}, see also \cite[Proposition 2]{CF2024} for a more general result.\smallskip
	
	As a consequence of Theorem \ref{Thm:CMS} \cite[Theorem A(iii)]{CMS}, we have
	
	\begin{Corollary}\label{Cor:(I:m)ComponentPolym}
		Let $I\subset S$ be a componentwise polymatroidal ideal. Then $(I:\m)$ is componentwise polymatroidal.
	\end{Corollary}
	\begin{proof}
		For all $j$, we have $(I:\m)_{\langle j\rangle}=(I_{\langle j+1\rangle}:\m)_{\langle j\rangle}=\textup{soc}(I_{\langle j+1\rangle})$, where the last equality follows from \cite[Proposition 1.4]{CHL}. Since $I_{\langle j+1\rangle}$ is polymatroidal, the assertion follows from Proposition \ref{Prop:HSn-1} and Theorem \ref{Thm:CMS}.
	\end{proof}
	
	Let $I$ be an ideal of $S$. The \textit{saturation} of $I$ is the ideal defined as
	$$
	I^{\textup{sat}}\ =\ \bigcup_{k\ge0}(I:\m^k).
	$$
	
	We have the ascending chain 
	$$
	I\subset(I:\m)\subset(I:\m^2)\subset\dots\subset(I:\m^k)\subset\cdots.
	$$
	
	Since $S$ is a Noetherian ring, there exists a smallest integer $\textup{sat}(I)$, called the \textit{saturation number} of $I$, such that $(I:\m^{k+1})=(I:\m^k)$ for all $k\ge\textup{sat}(I)$. Hence, $I^{\textup{sat}}=(I:\m^{\textup{sat}(I)})$. As a consequence of Corollary \ref{Cor:(I:m)ComponentPolym} we have
	
	\begin{Corollary}\label{Cor:Isat}
		Let $I\subset S$ be a componentwise polymatroidal ideal. Then $I^\textup{sat}$ is componentwise polymatroidal.
	\end{Corollary}
	\begin{proof}
		We can write $$I^{\textup{sat}}\ =\ (I:\m^{\textup{sat}(I)})=((\cdots(((I:\m):\m):\m):\cdots):\m),$$ where the colon is taken $\textup{sat}(I)$ times. Hence, the assertion follows by applying Corollary \ref{Cor:(I:m)ComponentPolym} $\textup{sat}(I)$ times.
	\end{proof}
	
	\section{Componentwise polymatroidality of $\HS_1(I)$}\label{sec10}
	
	Let $I\subset S$ be a componentwise polymatroidal ideal. We expect that $\HS_i(I)$ is again componentwise polymatroidal, see \cite[Question 9]{F-SCDP}. At the moment, we are only able to prove this expectation for the first homological shift ideal.
	\begin{Theorem}\label{Thm:HS-cp}
		Let $I\subset S$ be a componentwise polymatroidal ideal. Then $\HS_1(I)$ is componentwise polymatroidal.
	\end{Theorem}
	\begin{proof}
		Let $I=\bigoplus_{j\ge0}I_j$. Then $I_{\langle j\rangle}$ is a polymatroidal ideal for all $j$. We claim that $\HS_1(I)_{\langle j\rangle}=\HS_1(I_{\langle j-1\rangle})$ for all $j$. This fact together with Theorem \ref{Thm:CMS} implies the assertion. Let $L=\bigoplus_{j\ge0}L_j$ where $L_j$ is the $K$-vector space spanned by $\mathcal{G}(\HS_1(I_{\langle j-1\rangle}))$. We claim that $L$ is an ideal, and that $\HS_1(I)=L$. Once we acquire these claims, it follows that $\HS_1(I)_{\langle j\rangle}=L_{\langle j\rangle}=\HS_1(I_{\langle j-1\rangle})$ for all $j$, as desired.
		
		Let us show that $L$ is an ideal. We only need to prove that $\m L_{\langle j\rangle}\subset L_{\langle j+1\rangle}$ for all $j$. Let ${\bf 1}=(1,\dots,1)\in\mathbb{Z}_{\ge0}^n$. Since $I$ is an ideal, we have $\m I_{\langle j-1\rangle}\subset I_{\langle j\rangle}$, and consequently ${\bf deg}(I_{\langle j-1\rangle})+{\bf 1}\le{\bf deg}(I_{\langle j\rangle})$ for all $j$. Using this inequality and Proposition \ref{Prop:HS1-polym}, we have
		\begin{align*}
			\m L_{\langle j\rangle}\ &=\ \m\,\HS_1(I_{\langle j-1\rangle})\ =\ \m(\m I_{\langle j-1\rangle})^{\le{\bf deg}(I_{\langle j-1\rangle})}\\
			&\subset\ \m(I_{\langle j\rangle})^{\le{\bf deg}(I_{\langle j-1\rangle})}\ =\ (\m I_{\langle j\rangle})^{\le{\bf deg}(I_{\langle j-1\rangle})+{\bf 1}}\\
			&\subset\ (\m I_{\langle j\rangle})^{\le{\bf deg}(I_{\langle j\rangle})}\ =\ \HS_1(I_{\langle j\rangle})\\
			&=\ L_{\langle j+1\rangle}.
		\end{align*}
		
		It remains to prove that $\HS_1(I)=L$. Let $w\in L$ be a monomial of degree $j$. Then $w\in L_{\langle j\rangle}=\HS_1(I_{\langle j-1\rangle})$ and by \cite[Lemma 2.2]{FQ1} we have $w=\lcm(u,v)$ for some $u,v\in\mathcal{G}(I_{\langle j-1\rangle})\subset I$. Since, again by \cite[Lemma 2.2]{FQ1}, we have $$\HS_1(I)\ =\ (\lcm(f,g):\ f,g\in\mathcal{G}(I),\ f\ne g),$$ we see that $w\in\HS_1(I)$ and so $L\subset\HS_1(I)$.
		
		Conversely, let $w\in\HS_1(I)$ be a minimal monomial generator. Since $I$ has linear quotients \cite{F-SCDP}, by equation (\ref{eq:HerTak}) we can write $w=x_iu$ with $u\in\mathcal{G}(I)$ and there exists $v\in\mathcal{G}(I)$ with $v>u$ in the linear quotients order of $I$ such that $v:u=\lcm(v,u)/u=x_i$. By \cite[Lemma 2.1]{JZ}, we can assume that $\deg(v)\le\deg(u)$. Let $d=\deg(u)$, $d'=\deg(v)$ and set $v'=x_j^{d-d'}v$ for some $j\ne i$. Then $v',u\in\mathcal{G}(I_{\langle d\rangle})$ and moreover $\deg_{x_i}(v')=\deg_{x_i}(v)>\deg_{x_i}(u)$. By the symmetric exchange property applied to $I_{\langle d\rangle}$, we can find $s$ with $\deg_{x_s}(v')<\deg_{x_s}(u)$ such that $u'=x_i(u/x_s)\in\mathcal{G}(I_{\langle d\rangle})$. By \cite[Lemma 2.2]{FQ1}, we have $w=\lcm(u,u')\in\HS_1(I_{\langle d\rangle})=L_{\langle d+1\rangle}\subset L$. Hence $\HS_1(I)\subset L$ and this completes the proof.
	\end{proof}
	
	In general, products and powers of componentwise polymatroidal ideals are no longer componentwise polymatroidal, see \cite[Example 3.3]{BH2013}. On the other hand, it is expected \cite{BH2013,F-SCDP} that powers of a componentwise polymatroidal ideal $I\subset S$ are componentwise linear. Whether the homological shift algebra of $I$, or of any componentwise linear ideal, can be structured as a finitely generated $\mathcal{R}(I)$-module, or whether any of the results in this paper remains valid for componentwise polymatroidal ideals, remain open questions.\bigskip
	
	\noindent\textbf{Acknowledgment.}
	A. Ficarra was partly supported by INDAM (Istituto Nazionale di Alta Matematica), and also by the Grant JDC2023-051705-I funded by
	MICIU/AEI/10.13039/501100011033 and by the FSE.   D. Lu was supported by NSFC (No. 11971338).


\begin{thebibliography}{99}
		\bibitem{Ban24} S. Bandari, \textit{Polymatroidal ideals and linear resolution}, Journal of Algebraic Systems, {\bf 11}(2024) 147-153.
		
		\bibitem{BH2013} S. Bandari, J. Herzog, \textit{Monomial localizations and polymatroidal ideals}, Eur. J. Combin. {\bf 34} (2013) 752--763.
		
		\bibitem{Bay019} S. Bayati, \textit{Multigraded shifts of matroidal ideals}, Arch. Math., (Basel) {\bf 111} (2018), no. 3, 239--246.
		
		\bibitem{Bay2023} S. Bayati, \textit{A Quasi-additive Property of Homological Shift Ideals}. Bulletin of the Malaysian Mathematical Sciences Society, 2023, 46(3), p.111.
		
		\bibitem{BJT019} S. Bayati, I. Jahani, N. Taghipour, \textit{Linear quotients and multigraded shifts of Borel ideals}, Bull. Aust. Math. Soc. 100 (2019), no. 1, 48--57.
		
		\bibitem{CHL} L. Chu, J. Herzog, D. Lu, \textit{The socle module of a monomial ideal}, Rocky Mountain J. Math. 51 (2021), no. 3, 805--821.
		
		\bibitem{CMS} Y. Cid-Ruiz, J. P. Matherne, A. Shapiro, \textit{Syzygies of polymatroidal ideals}, 2025, preprint \url{arxiv.org/abs/2507.13153}
		
		\bibitem{Conca23} A. Conca, \textit{A note on the v-invariant}, Proceedings of the American Mathematical Society, 2024, 152(6), pp. 2349--2351
		
		\bibitem{CSTVV} S.M. Cooper, A. Seceleanu, S.O. Toh\u{a}neanu, M. Vaz Pinto, R.H. Villarreal, \textit{Generalized minimum distance functions and algebraic invariants of geramita ideals}. Advances in Applied Mathematics, 112:101940, 2020.
		
		
		\bibitem{CF1} M. Crupi, A. Ficarra, \textit{Very well--covered graphs by Betti splittings}, J. Algebra {\bf 629}(2023) 76--108. https://doi.org/10.1016/j.jalgebra.2023.03.033.
		
		\bibitem{CF2} M. Crupi, A. Ficarra, \textit{Very well-covered graphs via the Rees algebra}, Mediterranean Journal of Mathematics 21.4 (2024): 135.
		
		\bibitem{CF2024} M. Crupi, A. Ficarra, \textit{Cohen--Macaulayness of vertex splittable monomial ideals}, \textit{Mathematics}, 2024, 12(6): 912. Available at https://doi.org/10.3390/math12060912.
		
		\bibitem{F2} A. Ficarra, \textit{Homological shifts of polymatroidal ideals}, (2024), to appear in  Bull. math. Soc. Sci. Math. Roum., available at \url{https://arxiv.org/abs/2205.04163}
		
		\bibitem{F2023}  A. Ficarra, \textit{Simon conjecture and the $\v$-number of monomial ideals}, Collectanea Mathematica (2024): 1-16. https://doi.org/10.1007/s13348-024-00441-z
		
		\bibitem{FPack1} A. Ficarra, \textit{Homological Shift Ideals: Macaulay2 Package}, 2023, preprint \url{https://arxiv.org/abs/2309.09271}.
		
		\bibitem{F-SCDP} A. Ficarra, \textit{Shellability of componentwise discrete polymatroids}, Electron. J. Comb., Volume 32, Issue 1 (2025), P1.41.
		
		\bibitem{FH2023} A. Ficarra, J. Herzog, \textit{Dirac’s Theorem and Multigraded Syzygies}. Mediterr. J. Math. 20, 134 (2023). https://doi.org/10.1007/s00009-023-02348-8
		
		\bibitem{FM} A. Ficarra, S. Moradi, \textit{Symbolic powers of polymatroidal ideals}, J. Pure Appl. Algebra, 229, 10(2025), 108082, https://doi.org/10.1016/j.jpaa.2025.108082
		
		\bibitem{FM1} A. Ficarra, S. Moradi, \textit{Complementary edge ideals}, (2025), preprint \url{https://arxiv.org/abs/2508.10870}
		
		\bibitem{FQ} A. Ficarra, A.A. Qureshi, \textit{The homological shift algebra of a monomial ideal}, (2024), preprint \url{https://arxiv.org/abs/2412.21031}.
		
		\bibitem{FQ1} A. Ficarra, A.A. Qureshi, \textit{Edge ideals and their asymptotic syzygies}, J. Pure Appl. Algebra, 229, 10(2025), 108079, https://doi.org/10.1016/j.jpaa.2025.108079
		
		\bibitem{FS2} A. Ficarra, E. Sgroi, \textit{Asymptotic behaviour of the $\v$-number of homogeneous ideals}, 2023, preprint \url{https://arxiv.org/abs/2306.14243}.
		
		\bibitem{FSPack} A. Ficarra, E. Sgroi, \texttt{VNumber}, \textit{Macaulay2 Package} available at \url{https://github.com/EmanueleSgroi/VNumber}, 2024.
		
		\bibitem{FSPackA} A. Ficarra, E. Sgroi, \textit{Asymptotic behaviour of integer programming and the $\v$-function of a graded filtration}, J. Algebra Its Appl., 2025, https://doi.org/10.1142/S0219498826502361
		
		\bibitem{FVT2007} C.A. Francisco, A. Van Tuyl, \textit{Some families of componentwise linear monomial ideals}, Nagoya Math. J. 187 (2007), 115--156.
		
		\bibitem{GDS} D.~R.~Grayson, M.~E.~Stillman. {\em Macaulay2, a software system for research in algebraic geometry}. Available at \url{http://www.math.uiuc.edu/Macaulay2}.
		
		\bibitem{HHBook} J.~Herzog, T.~Hibi, \emph{Monomial ideals}, Graduate texts in Mathematics {\bf 260}, Springer, 2011.
		
		\bibitem{HHV} J. Herzog, T. Hibi, M. Vladoiu, \textit{Ideals of fiber type and polymatroids}, Osaka J. Math. 42 (4) (2005) 807--829.
		
		\bibitem{HMRZa} J. Herzog, S. Moradi, M. Rahimbeigi, G. Zhu, \textit{Homological shift ideals}. Collect. Math. {\bf 72} (2021), 157--74.
		
		\bibitem{HMRZb} J. Herzog, S. Moradi, M. Rahimbeigi, G. Zhu, \textit{Some homological properties of borel type ideals}, Comm. Algebra {\bf 51} (4) (2023) 1517--1531.
		
		\bibitem{HRV} J. Herzog, A. Rauf, M. Vladoiu, \textit{The stable set of associated prime ideals of a polymatroidal ideal}, J. Algebr. Comb. {\bf 37} (2013), 289-312.
		
		\bibitem{ET} J.~ Herzog, Y. Takayama, \textit{Resolutions by mapping cones}, in: The Roos Festschrift volume Nr.2(2), Homology, Homotopy and Applications {\bf 4}, (2002), 277--294.
		
		\bibitem{HS2006} C. Huneke, I. Swanson. \textit{Integral closure of ideals, rings, and modules}. Vol. 336. London Mathematical Society Lecture Note Series. Cambridge University Press, Cambridge, 2006.
		
		\bibitem{JZ} A.S. Jahan, X. Zheng, \textit{Ideals with linear quotients}, J. Combin. Theory Ser. A {\bf 117} (2010), 104--110. MR2557882 (2011d:13030)
		
		\bibitem{KNQ} K. Khashyarmanesh, M. Nasernejad, A. A. Qureshi, \textit{On the matroidal path ideals}, J. Algebra Appl. (2023), 2350227.
		
		\bibitem{LW} D. Lu, Z. Wang, \textit{The resolutions of generalized co-letterplace ideals and their powers}, Journal of Algebra, 2025, 673, 321--350.
		
		\bibitem{MN21} A. Mafi, D. Naderi, \textit{Linear resolutions and polymatroidal ideals}, Proceedings-Mathematical Sciences 131 (2021), 1--15.
		
		\bibitem{RS} A. Roy, K. Saha, \textit{On the homological shifts of cover ideals of Cohen-Macaulay graphs}, (2025), preprint \url{https://arxiv.org/abs/2506.01810}
		
		\bibitem{TBR24} N. Taghipour, S. Bayati, F. Rahmati, \textit{Homological linear quotients and edge ideals of graphs}. Bulletin of the Australian Mathematical Society (2024): 1-12.
	\end{thebibliography}
\end{document}